\documentclass[11pt]{amsart}

\usepackage{amsmath,amssymb,latexsym,cite,
mathrsfs}
\usepackage[utf8]{inputenc}
\usepackage{xcolor,graphicx}
\usepackage[colorlinks=true,urlcolor=blue,
citecolor=red,linkcolor=blue,linktocpage,pdfpagelabels,
bookmarksnumbered,bookmarksopen]{hyperref}
\usepackage[english]{babel}

\usepackage{mathrsfs}

\usepackage[left=2.9cm,right=2.9cm,top=2.8cm,bottom=2.8cm]{geometry}
\usepackage[hyperpageref]{backref}

\makeatletter
\providecommand\@dotsep{5}
\def\listtodoname{List of Todos}
\def\listoftodos{\@starttoc{tdo}\listtodoname}
\makeatother

\numberwithin{equation}{section}

\everymath{\displaystyle}

\newtheorem{theorem}{Theorem}[section]
\newtheorem{proposition}[theorem]{Proposition}
\newtheorem{lemma}[theorem]{Lemma}

\newtheorem{claim}[theorem]{Claim}

\newtheorem{remark}{Remark}

\newcommand{\R}{\mathbb{R}}

\begin{document}

\title{A Critical Neumann problem with anisotropic $p-$Laplacian }
\author{ Gustavo F. Madeira, Olímpio H. Miyagaki$^{\Diamond}$, Alânnio B. Nóbrega$^{\star}$ }

\address[Alânnio B. Nóbrega]
{\newline\indent Unidade Acadêmica de Matemática
	\newline\indent 
	Universidade Federal de Campina Grande
	\newline\indent
	CEP: 58429-900- Campina Grande - PB, Brazil} 
\email{alannio@mat.ufcg.edu.br}

\address[Gustavo F. Madeira, Olímpio H. Miyagaki]
{\newline\indent Departamento de Matemática
\newline\indent 
Universidade Federal de São Carlos
\newline\indent
CEP: 13565-905-São Carlos-SP,Brazil} 
\email{gfmadeira@ufscar.br, olimpio@ufscar.br}

\pretolerance10000

\begin{abstract} 
We are concerned with the existence of solution of the problem 
$$\left\{\begin{array}{lcl}
	-\Delta ^H_pu+|u|^{p-2}u=\lambda|u|^{q-2}u+ |u|^{p^*-2}u\quad \mbox{in\,\,}\Omega,\\
	u>0\quad \mbox{in\,\,}\Omega,\\
	a(\nabla u)\cdot \nu =0\quad \mbox{on\,\,}\partial \Omega,
\end{array}\right. \eqno{(P)}$$ where  $\Delta ^H_pu=\mbox{div\,}(a(\nabla u))$, with $a(\xi)=H^{p-1}(\xi)\nabla H(\xi),\, \xi \in \R^N,$ $N\geqslant3,$ is the anisotropic $p$-Laplacian with $1<p<N$, $\lambda>0$ is a parameter, and   $p < q<p^*=pN/(N-p)$.  Further, $\Omega \subset \Sigma$ is a $C^1$ bounded domain inside a convex open cone $\Sigma$ in $\R^N$ with $\partial \Omega \cap \partial \Sigma$ being a $C^1$-manifold, and $\nu$ is the unit outward normal to $\partial \Omega$. To succeed with a variational approach, where the strong convergence of a bounded (PS) subsequence needs to be proved, one has to deal with anisotropic norms in the absence of a Tartar's type inequality, unlike the isotropic $p$-Laplace case. This is overcome by proving the a.e. convergence of its gradients. Furthermore, the solution of $(P)$ is shown to belong to $C^{1,\alpha}(\Omega)$, and is strictly positive in $\Omega$. Such conclusions are achieved from classical elliptic regularity theory and a Harnack inequality, since the solution of $(P)$ is bounded. This in turn is a consequence of a result in this paper which ensures that any $W^{1,p}$-solution of critical Neumann problems with the anisotropic $p$-Laplacian operator on bounded Lipschitz domains in $\mathbb{R}^N$ $(N\geqslant3)$ is bounded.
  
\end{abstract}
\thanks{$^{\Diamond}$ O. H. Miyagaki  was  supported in part by  CNPq Proc. $\#303256/2022-2.$}
\thanks{$^{\star}$ A. B. Nóbrega was  partially supported by grant  $\#3177/2021$ FAPESQ-PB and by Projeto Universal FAPESQ-PB $\#3031/2021$}

\maketitle

\section{Introduction and Main Result}

Elliptic problems of the form
\begin{equation}\label{P1}
\left\{\begin{array}{lcl}
	-\Delta u=f(x,u)+ u^{2^*-1}\,\quad \mbox{in\,\,}\Omega,\\
	u>0\,\quad \mbox{in\,\,} \Omega\\
	u =0\,\quad \mbox{on\,\,} \partial \Omega,
\end{array}\right.
\end{equation}have been intensively studied since the second half of last century. A common situation in \eqref{P1} stands to the case where $\Omega$ is a bounded domain in $\R^N$ with a $C^1$ boundary, $N \ge 3$, $2^*=2N/(N-2)$ is the critical Sobolev exponent and $f:\Omega \times \R \rightarrow \R$ is a Caratheodory lower order perturbation of $u^{2^*-1}$ at infinity. The corresponding literature on this setting is very large, and classical results have been obtained. For instance, Pohozaev  \cite{P} have proved that if $\Omega $ is a star-shaped domain, then problem \eqref{P1} has no solution in the purely critical power case (i.e., with $f(x,u)=0$ in \eqref{P1}). Later on, Brezis and Nirenberg \cite{BN} have proved that if $f(x,u)$ is a suitable lower order perturbation of $u^{2^*-1}$ at infinity, then there exists a solution of \eqref{P1}. Capozzi, Fortunato and Palmieri \cite{CFP} have extended some results in \cite{BN}, Cerami, Fortunato and Struwe \cite{CSF} have proved the multiplicity of solution of \eqref{P1}. Cerami, Solimini and Struwe \cite{CSS} have studied the existence of nodal and radial solutions of \eqref{P1}, and Garcia and Peral \cite{AA} have considered \eqref{P1} with the $p$-Laplace, among many other works, just to quote a few. 

In \cite{Wang}, Wang has considered the counterpart of \eqref{P1} with Neumann boundary condition, showing the existence of solution of 
 \begin{equation}\label{prob-Neumann-Wang1}
 	\left\{\begin{array}{lcl}
 		-\Delta u=f(x,u)+ u^{2^*-1}\,\quad \mbox{in\,\,}\Omega,\\
 		u>0\,\quad \mbox{in\,\,}\Omega,\\
 		\partial u/\partial \nu + \alpha(x)u =0\,\quad \mbox{on\,\,}\partial\Omega.
 	\end{array}\right.
 \end{equation}
As ingredients in \eqref{prob-Neumann-Wang1}, $f(x,u)$ is a suitable lower order perturbation of $u^{2^*-1}$ at infinity, $\nu$ is the outward normal to $\partial \Omega,$ and $\alpha\in L^{\infty}(\partial\Omega)$, $\alpha \not\equiv0,$ is a non-negative function.
Later on, Wang \cite{Wang2} has studied a critical Neumann problem with the $p$-Laplace operator. Existence results for Neumann problems involving perturbations by concave terms have been treated by Chabrowski \cite{Ch} and Chabrowski and Yang \cite{CY}. Further, critical elliptic systems under Neumann boundary condition have been considered by Chabrowski and Yang \cite{CYs}, and Kou and An \cite{Kou}. 

Recently,  Ciraolo, Figalli and Roncoroni \cite{Figalli 1}, with the goal of generalizing the classification of the positive solutions of $\Delta_pu + u^{p^*-1}=0$ in the whole space to critical anisotropic $p$-Laplace equations in convex cones, have studied the problem
\begin{equation}\label{Figalli}
\left\{\begin{array}{lcl}
	\Delta ^H_pu + u^{p^*-1}=0\,\quad \mbox{in\,\,}\Sigma,\\
	u>0\,\quad \mbox{in\,\,}\Sigma,\\
	a(\nabla u)\cdot \nu =0\,\quad \mbox{on\,\,} \partial\Sigma,\\
	u\in\mathcal{D}^{1,p}(\Sigma).
\end{array}\right.
\end{equation}In \eqref{Figalli}, the operator $\Delta ^H_pu=\mbox{div\,}(a(\nabla u))$, where $a(\xi)=H^{p-1}(\xi)\nabla H(\xi)$ for all $\xi \in \R^N$, and $H(\cdot)$ is a norm in $\R^N,$ is  the \textit{anisotropic} or \textit{Finsler} $p$-\textit{Laplacian}. Further, $\Sigma$ is a convex open cone in $\R^N$ given by $\Sigma=\{tx\::\: x \in \omega, t\in (0,\infty)\}$ for some open domain $\omega\subset \mathbb{S}^{N-1}$, $\nu$ is the outward normal to $\partial\Sigma$, and $\mathcal{D}^{1,p}(\Sigma)=\{u\in L^{p^*}(\Sigma)\::\:\nabla u\in L^{p}(\Sigma)\}$. Besides a complete classification of the solutions of \eqref{Figalli}, the authors establish in \cite{Figalli 1} a corresponding sharp version of a Sobolev's inequality in cones, namely,
\begin{equation}\label{1}
	\| u\|^{p}_{L^{p^*}(\Sigma)}\le S_{\Sigma, H}^{-1}\|H(\nabla u)\|^{p}_{L^p(\Sigma)}.
\end{equation}Moreover, they prove that the extremals of \eqref{1} are of the form
\begin{equation}\label{extrem-funct}
U^{H}_{\lambda, x_0}(x):= \left(\frac{\lambda^{\frac{1}{p-1}}\left(n^{\frac{1}{p}}\left(\frac{n-p}{p-1}\right)^{\frac{p-1}{p}}\right)}{\lambda^{\frac{p}{p-1}}+\hat{H}_0(x-x_0)^{\frac{p}{p-1}}}\right)^{\frac{n-p}{p}}
\end{equation}for some $\lambda>0,$ where $\hat{H}_0(\zeta)=H_0(-\zeta),$ and $H_0$ denotes the dual norm associated to $H$, given by $H_0(\zeta)=\sup_{H(\xi)=1}\zeta\cdot\xi$ for all $\zeta\in\R^N.$ This means that, by setting 
\begin{equation}\label{2}
	S_{\Sigma, H}=\inf_{\stackrel{u \in W^{1,p}(\Sigma)}{u \neq 0}}  \frac{\|H(\nabla u)\|^{p}_{L^p(\Sigma)}}{\| u\|^{p}_{L^{p^*}(\Sigma)}},
\end{equation}one has $S_{\Sigma, H}$ is actually the best constant in \eqref{1}, which is attained by the functions $U^{H}_{\lambda, x_0}$. 

 Anisotropic or Finsler $p$-Laplacian type operators supplied with Neumann, Robin, or mixed boundary conditions have recently been studied by Ciraolo, Corso and Roncoroni \cite{CCR}, Dipierro, Poggesi and Valdinoci \cite{DipPogVal}, Montoro and Sciunzi\cite{MS}, see also the references quoted therein.

The aim of this paper is to extend some of the results obtained in Wang \cite{Wang} to the framework of anisotropic $p$-Laplace type operators. To be more precise, we shall study the critical problem 
$$\left\{\begin{array}{lcl}
	-\Delta ^H_pu+|u|^{p-2}u=\lambda|u|^{q-2}u+ |u|^{p^*-2}u\quad \mbox{in\,\,}\Omega,\\
	u>0\quad \mbox{in\,\,} \Omega,\\
	a(\nabla u)\cdot \nu =0\quad \mbox{on\,\,} \partial \Omega,
\end{array}\right. \eqno{(P)}$$ by using variational techniques, obtaining a positive solution of $(P)$ (see Theorem \ref{Main}). Regarding problem $(P)$, it is worth mentioning that Montoro and Sciunzi\cite{MS} established the Pohozaev identity for the anisotropic $p$-Laplacian, what allowed them to obtain a non-existence result for a Dirichlet counterpart to $(P)$. Further, Ciraolo, Corso and Roncoroni \cite{CCR} proved there can not exist a non-constant solution to the anisotropic $p$-Laplacian Neumann problem
$$\left\{\begin{array}{lcl}
	-\Delta ^H_pu+|u|^{p-2}u=|u|^{p^*-2}u\quad \mbox{in\,\,}\Omega,\\
	u>0\quad \mbox{in\,\,} \Omega,\\
	a(\nabla u)\cdot \nu =0\quad \mbox{on\,\,} \partial \Omega,
\end{array}\right. $$which corresponds to the case $\lambda=0$ in $(P).$ To obtain a solution of $(P)$ employing variational methods, we set the functional $J:W^{1,p}(\Omega) \rightarrow \R$ defined by 
\begin{equation}\label{func-J}
J(u)=\frac{1}{p}\int_{\Omega} [H(\nabla u)^p+|u|^p]\,dx-\frac{\lambda}{q}\int_{\Omega}(u^+)^q\,dx-\frac{1}{p^*}\int_{\Omega}(u^+)^{p^*}\,dx.
\end{equation}By a (weak) solution of $(P)$ we mean a critical point of $J$, i.e., $u \in W^{1,p}(\Omega)$ satisfying
$$\int_{\Omega} \left[H(\nabla u)^{p-1}\langle \nabla H(\nabla u), \nabla v\rangle+(u^+)^{p-1}v \right]\,dx - \lambda\int_{\Omega} (u^+)^{q-1} v \,dx-\int_{\Omega}(u^+)^{p^*-1}v\,dx=0,$$
for all $v \in W^{1,p}(\Omega).$ The main result in this paper with respect to $(P)$ reads as
\begin{theorem}\label{Main}
	Let $ p \ge 2$ and N $\ge p^2$ or  $N=p^2-p+1$. Then there exists at least one non-trivial solution of $(P)$ for $\lambda$ suitably large. Moreover, $u\in C^{1,\alpha}(\Omega)$ and $u>0$ in $\Omega$. 
\end{theorem}An important step in the proof is to show that the weak limit of a $(PS)$ sequence is a critical point of $J.$ If the right-hand side of $(P)$ gives rise to the isotropic norm $(\textstyle\int_{\Omega}[|\nabla u|^p+|u|^p]dx)^{1/p}$ in $W^{1,p}(\Omega),$ i.e., if $-\Delta ^H_p$ is the standard $p$-Laplace, the so called Tartar's type inequalities (see \cite{Simo}) are a very helpful tool to accomplish such a step, see \cite{Wang,Wang2}. That seems not be the case for general anisotropic norms $(\textstyle\int_{\Omega}[H^p(\nabla u)+|u|^p]dx)^{1/p}$ in $W^{1,p}(\Omega).$ Actually, to our knowledge there is no Tartar's type inequalities available to the anisotropic $p$-Laplace operator $-\Delta ^H_p.$ To overcome this issue, we first establish the strict monotonicity of $-\Delta ^H_p$ (see Proposition \ref{H monot}). So we use Boccardo-Murat's \cite{Boccardo} result to set the almost everywhere gradient convergence of a weakly convergent $(PS)$ sequence of $J$, and obtain a mountain pass critical point $u$ of $J.$ Then we adapt the approach in Wang \cite{Wang} to determine the energy sub-levels of $J$ ensuring compactness, which the mountain pass level of $u$ belongs to, so $u$ is a weak solution of $(P)$. Since $u$ is bounded, what is established in Theorem \ref{thm-L-infinito} below, one can invoke elliptic regularity theory and a Harnack inequality to conclude $u\in C^{1,\alpha}(\Omega)$ is a positive solution of $(P)$. 

This paper is organized as follows. In section 2 we set the framework and derive some properties of the anisotropic $p$-Laplace operator. Indeed, we establish its strict monotonicity, and a convergence result for the gradients of weak solutions of the equation in $(P)$. Furthermore, we prove the $L^\infty$-regularity of solutions of a general class of critical Neumann problems involving anisotropic $p$-Laplace on bounded smooth (Lipschitz) domains. In Section 3 we prove several preparatory lemmas in order to, in Section 4, complete the proof of Theorem \ref{Main}.

\section{The anisotropic $p$-Laplace operator $-\Delta ^H_p$: framework and some properties}
\subsection{Framework and strict monotonicity}
In this section we set the framework for the anisotropic $p$-Laplace operator $-\Delta ^H_p$, and some related properties. For additional details, we refer \cite{B,CMV1,CMV2,{Fari-Vald},F,WX} and the references therein. Let $H:\R^N \rightarrow \R$ be a non-negative convex even function of class $C^2(\R^N\setminus  \{0\}),$ satisfying

\bigskip

\noindent $(H_1)$ $H(t\xi)=|t|H(\xi), \, \forall\, t\in\, \R\,\,\mbox{and}\,\,\forall\,\xi\, \in\, \R^N,$ 

\bigskip


\noindent $(H_2)$ $H(\xi)>0, \, \forall\,\xi\neq 0.$

\bigskip

 By homogeneity, there exist constants $0<C_1 \le C_2<\infty$ with
\begin{equation}\label{norm-eqv}
C_1|\xi|\le H(\xi) \le C_2 |\xi|,\ \forall\,\xi\, \in\, \R^N.
\end{equation}We will also assume 

\bigskip

\noindent $(H_3)$ $\{\xi \in \R^N; H(\xi) < 1\}$ is uniformly convex;

\medskip

\noindent where uniform convexity means the principal curvatures of the boundary are positive
and bounded away from zero.

\bigskip

From \cite{CMV2} (Proposition 3.1) we have $(H_3)$ is equivalent, for any $p>1$, to $H^p$ to be strictly convex, and there exist positive constants $\gamma, \Gamma$ such that
$$(Hess(H^p)(\xi))_{ij}\zeta_i\zeta_j\ge \gamma|\xi|^{p-2}|\zeta|^2,\qquad \sum_{i,j=1}^{N}\left|(Hess(H^p)(\xi))_{ij}\right|\le \Gamma|\xi|^{p-2}, $$
for any $\xi \in \R^N\setminus\{0\}$ and $\zeta\in \R^N.$ Let $H_0:\R^N \rightarrow \R$ be given by
$$H_0(x)= \sup_{\{\xi :H(\xi)\leq1\}}\langle x, \xi\rangle.$$
It holds that $H_0\in C^2(\R^N\setminus \{0\})$ is a convex, non-negative, homogeneous function. The function $H_0$ is dual of $H$ in the sense that
$$H_0(x)=\sup_{\xi \neq 0}\frac{\langle x,\xi\rangle}{H(\xi)}\quad\mbox{and}\quad H(x)=\sup_{\xi \neq 0}\frac{\langle x,\xi\rangle}{H_0(\xi)}.$$
In next lemma we collect some properties of  $H$ that will be useful along the paper.
\begin{lemma}[See \cite{Zhou-zhou}]\label{H-prop}
For all $x,y,\xi \in\mathbb{R}^N, t\in\mathbb{R},$ the following assertions hold.
	\begin{enumerate}
		\item[(i)] $|H(x)-H(y)|\le H(x+y) \le H(x)+H(y);$
		\item[(ii)] $\frac{1}{C}\le |\nabla H(x)|\le C,$ and $\frac{1}{C}\le |\nabla H_0(x)|\le C$ for some $C>0$ and any $x \neq 0;$
		\item[(iii)] $\langle x,\nabla H(x)\rangle=H(x),$ $\langle x,\nabla H_0(x)\rangle=H_0(x)$ for any $x \neq 0;$
		\item[(iv)] $H(\nabla H_0(x))=1,$ $H_0(\nabla H(x))=1$ for any $x \neq 0;$
		\item[(v)] $H_0(x)H_{\xi}(\nabla H_0(x))=x$ for any  $x \neq 0;$
		\item[(vi)] $H_{\xi}(t \xi)=sgn(t)H_{\xi}( \xi)$ for any  $\xi \neq 0$ and $t \neq 0$.
	\end{enumerate}
\end{lemma} 

\noindent Under the hypotheses above, we consider for $p\geq2$ the anisotropic or Finsler $p$-Laplacian 
\begin{equation}\label{def-op-anisot}
\Delta ^H_pu=\mbox{div\,}(a(\nabla u))
\end{equation}where $a(\xi)=H^{p-1}(\xi)\nabla H(\xi)$ for all $\xi \in \R^N$, and $H$ satisfies $(H_1)-(H_3)$. Next result shows that $-\Delta ^H_p$ is strictly monotone, a convenient property not only for the purposes of this work. 
\begin{proposition}\label{H monot}
Let $p > 1$ and let $H:\R^N \rightarrow \R$ be a non-negative convex function of class $C^2(\R^N\setminus \{0\})$ satisfying $(H_1)-(H_2)$. Then there holds
\begin{equation}\label{monot}
	\langle H^{p-1}(\xi)\nabla H(\xi)- H^{p-1}(\eta)\nabla H(\eta),\,\xi-\eta\rangle>0,\quad\mbox{for all}\quad \xi\neq\eta.
\end{equation}
\end{proposition}
\begin{proof}Since $
	\nabla(H^p(\xi))=pH^{p-1}(\xi)\nabla H(\xi) $ for all $\xi \in\mathbb{R}^N,$ it follows that
$$
\begin{aligned}
	\langle H^{p-1}(\xi)\nabla H(\xi)- H^{p-1}(\eta)\nabla H(\eta),\,\xi-\eta\rangle =\, &\frac{1}{p}\sum_{i=1}^{N}\left( \partial_{\xi_i}(H^p(\xi))-\partial_{\xi_i}(H^p(\eta))\right)(\xi_i-\eta_i)\\
	= &-\frac{1}{p}\sum_{i=1}^{N}\int_{0}^{1}\frac{d}{dt}\partial_{\xi_i}(H^p(\xi+t(\eta-\xi)))dt(\xi_i-\eta_i).
\end{aligned}
$$Therefore we have

$$\begin{aligned}
	\langle H^{p-1}(\xi)\nabla H(\xi)- H^{p-1}(\eta)&\nabla H(\eta),\,\xi-\eta\rangle\\ &=-\frac{1}{p}\int_{0}^{1}\sum_{i=1}^{N}\left(\sum_{j=1}^{N}\partial^2_{\xi_j\xi_i}(H^p)(\xi+t(\eta-\xi))(\eta_j-\xi_j)\right)dt(\xi_i-\eta_i)\\
	&=\frac{1}{p}\int_{0}^{1}\sum_{i,j=1}^{N}\partial^2_{\xi_j\xi_i}(H^p)(\xi+t(\eta-\xi))(\xi_j-\eta_j)(\xi_i-\eta_i)dt.
\end{aligned}
$$Now set the function $g:[0,1]\rightarrow\mathbb{R}$ defined by
	$$g(t)=\frac{1}{p}\int_{0}^{1}\sum_{i,j=1}^{N}\partial^2_{\xi_j\xi_i}(H^p)(\xi+t(\eta-\xi))(\xi_j-\eta_j)(\xi_i-\eta_i)dt.$$
Since the line $\xi+t(\eta-\xi)$ passes through the origin at most one time for all $\xi \neq \eta$, by $(H_3)$ we must have $\textstyle\int_{0}^{1}g(t)dt > 0.$ Hence \eqref{monot} holds, and the proof is complete.
\end{proof}

\subsection{Almost Everywhere Convergence of the gradients}

The following result, due to Boccardo and Murat \cite{Boccardo}, will be important along the proof of Theorem \ref{Main}.
\begin{lemma}[See \cite{Boccardo}]\label{Bo1}
	Let $a:\Omega \times \R\times \R^N \rightarrow \R^N$ is a Carathéodory function satisfying
	\begin{eqnarray}
	&|a(x,s,\zeta)|\le c(x)+k_1|s|^{p-1}+k_2|\zeta|^{p-1}, \label{Bo2}\\
	&\left[a(x,s,\zeta)-a(x,s,\eta)\right]\left[\zeta-\eta\right]>0, \,\,\,\,\,\,\,\,\,\,\,\,\,\,\,\, \label{Bo3}\\
&\frac{a(x,s,\zeta)\zeta}{|\zeta|+|\zeta|^{p-1}}\to +\infty,\,\,\mbox{as}\,\,|\zeta|\to \infty, \,\,\,\,\,\,\,\,\,\,\,\,\,\,\,\,\,\,\,\,\label{Bo4}
	\end{eqnarray}
for a.e. $x \in \Omega$, for all $s\in \R$, for all $\zeta,\eta \in \R^N$ with $\zeta\neq\eta,$ where $c(x)$ belongs to $L^{p'}(\Omega),$ $c\ge 0,$ and $k_1,k_2 \in \R^{+}.$ Suppose 
\begin{equation}\label{Bo5}
	-div(a(x,u_{n},\nabla u_{n}))=f_n+g_n,\,\,\mbox{in}\,\,\mathscr{D}'(\Omega),
\end{equation} where:
\begin{eqnarray}
	u_n \rightharpoonup u\,\,\mbox{weakly in}\,\,W^{1,p}(\Omega),\,\,\mbox{strongly in}\,\,L^{p}_{loc}(\Omega)\,\,\mbox{and}\,\,a.e.\,\,\mbox{in}\,\,\Omega; \label{Bo6}\\
	f_n \rightarrow f\,\,\mbox{strongly in}\,\,W^{-1,p'}(\Omega);\\
	|\langle g_n,\varphi\rangle|\le C_k \|\varphi\|_{L^{\infty}(\Omega)}\,\,\mbox{for any}\,\,\varphi \in \mathscr{D}(\Omega)\,\,\mbox{with}\,\,\mbox{\textnormal{supp\,}}(\varphi)\subset K, \label{Bo7}
\end{eqnarray}
where $C_k$ is a constant which depends on the compact set $K$.
Then,
\begin{equation}\label{Bo8}
	\nabla u_n \rightarrow \nabla u\,\,\mbox{strongly in}\,\,(L^q(\Omega))^{N}\,\,\mbox{for any}\,\,q<p.
\end{equation}As a consequence, passing to a subsequence, one has 
\begin{equation}\label{Bo-ae}
	\nabla u_n \rightarrow \nabla u\,\,\mbox{a.e. in}\,\,\Omega.
\end{equation}
\end{lemma} To apply previous lemma to problem $(P)$ we need to prove the following

\begin{lemma}\label{Bo9}
	The operator $\Delta ^H_p$ satisfies $(\ref{Bo2})-(\ref{Bo4})$.
\end{lemma}
\begin{proof}
From Lemma \ref{H-prop}(ii), \eqref{norm-eqv} and \eqref{def-op-anisot}, there exists $\tilde{C} >0$ such that 
$$
\begin{array}{lcl}
	|a(\xi)|=|H^{p-1}(\xi)\nabla H(\xi)|=H^{p-1}(\xi)|\nabla H(\xi)| \le CH^{p-1}(\xi) \le \tilde{C}|\xi|^{p-1}
\end{array}
$$
for all $,\xi \in \R^N$, so $a$ satisfies $(\ref{Bo2})$. Condition $(\ref{Bo3} )$ follows from Proposition \ref{H monot}. Now since
$$	\frac{a(\zeta)\zeta}{|\zeta|+|\zeta|^{p-1}}=\frac{H^{p-1}(\zeta)\nabla H(\zeta)\zeta}{|\zeta|+|\zeta|^{p-1}}\quad\forall\zeta,\eta \in \R^N,$$ 
from Lemma \ref{H-prop}(iii) and (\ref{norm-eqv}), we have
$$	\frac{a(\zeta)\zeta}{|\zeta|+|\zeta|^{p-1}}=\frac{H^{p}(\zeta)}{|\zeta|+|\zeta|^{p-1}}\ge \frac{C_1^p|\zeta|^p}{|\zeta|+|\zeta|^{p-1}}\ge C_1^p\min\{|\zeta|^{p-1},|\zeta|\}.$$
Thus $a(\zeta)\zeta\big/(|\zeta|+|\zeta|^{p-1})\rightarrow+\infty,$ as $|\zeta| \to +\infty,$ proving (\ref{Bo4}).
\end{proof}

\begin{proposition}\label{Msec2}
	Let $(u_n)\subset W^{1,p}(\Omega)$ be a sequence such that
$$-\mbox{\textnormal{div\,}}(H^{p-1}(\nabla u_n)\nabla H(\nabla u_n)) +|u_n|^{p-2}u_n=\lambda|u_n|^{q-2}u_n+ |u_n|^{p^*-2}u_n\,\,\,\, \mbox{in}\,\, \Omega,$$
and $u_n \rightharpoonup u$ weakly in $W^{1,p}(\Omega)$, $u_n \rightarrow u$ strongly in $L^{p}_{loc}(\Omega)$ and $u_n \rightarrow u$ a.e. in $\Omega.$
Then $\nabla u_n \rightarrow \nabla u\,\,\mbox{strongly in}\,\,(L^q(\Omega))^{N}\,\,\mbox{for any}\,\,q<p.$ By passing to a subsequence, if necessary, one has $\nabla u_n \rightarrow \nabla u$ a.e. in $\Omega.$
\end{proposition}
\begin{proof}
The claims will follow by applying Lemma \ref{Bo1}. Let $f_n=\lambda |u_n|^{q-2}u_n-|u_n|^{p-2}u_n$ and $g_n=|u_n|^{p^*-2}u_n$. Given $v \in W^{1,p}(\Omega),$ by setting $f=\lambda |u|^{q-2}u-|u|^{p-2}u$ we have
$$
\begin{aligned}
\langle f_n-f, v\rangle_{W^{-1,p'}}&=\lambda\int_{\Omega}\left(|u_n|^{q-2}u_n- |u|^{q-2}u\right)vdx-\int_{\Omega}\left(|u_n|^{p-2}u_n-|u|^{p-2}u\right)vdx\\
&\leq C\bigg[\||u_n|^{q-2}u_n\!- |u|^{q-2}u\|_{L^{\frac{q}{q-1}}(\Omega)}\!\!\!-\||u_n|^{p-2}u_n\!- |u|^{p-2}u\|_{L^{\frac{p}{p-1}}(\Omega)}\bigg]\|v\|_{W^{1,p}(\Omega)}.
\end{aligned}
$$
Thus since $u_n \rightharpoonup u$ in $W^{1,p}(\Omega)$ and $p \le q \le p^*,$ we have
	\begin{equation}\label{boh1}
	\|f_n-f\|_{W^{-1,p'}}\leq C\bigg[\||u_n|^{q-2}u_n\!- |u|^{q-2}u\|_{L^{\frac{q}{q-1}}(\Omega)}\!\!\!-\||u_n|^{p-2}u_n\!- |u|^{p-2}u\|_{L^{\frac{p}{p-1}}(\Omega)}\bigg] \to 0.
\end{equation}
Now let $\varphi \in \mathscr{D}(\Omega)$ with $\mbox{supp\,}(\varphi)\subset K$. Then it holds
$$|\langle g_n, \varphi\rangle|\le \int_{K}|u_n|^{p^*-1}|\varphi|dx\le \|\varphi\|_{\infty}\int_{K}|u_n|^{p^*-1}dx.$$
Using that $u_n \rightharpoonup u$ in $W^{1,p}(\Omega)$, passing to a subsequence if necessary,  we have $u_n \rightarrow u$ in $L^{p^*-1}(\Omega)$, so there exists a constant $C>0$ (uniform with respect to $K$)  such that  
\begin{equation}\label{boh2}
|\langle g_n, \varphi\rangle|\le C \|\varphi\|_{\infty}.
\end{equation}Therefore from (\ref{boh1}) and (\ref{boh2}) the desired conclusions will follow from Lemma \ref{Bo1}. 
\end{proof}

\subsection{Boundedness of solution of critical anisotropic $p$-Laplace Neumann problems}
In this section we are interested in general Neumann problems of the form
\begin{equation}\label{prob-Neumann-geral}
\left\{\begin{array}{lcl}
	-\Delta ^H_pu=f(x,u)\quad \mbox{in\,\,}\Omega,\\
	a(\nabla u)\cdot \nu =0\quad \mbox{on\,\,} \partial \Omega,
\end{array}\right.
\end{equation}involving the anisotropic $p$-Laplace operator $-\Delta ^H_p$ given by \eqref{def-op-anisot}, and sources $f$ with critical growth. More precisely, we prove (see Theorem \ref{thm-L-infinito} below) that weak solutions of \eqref{prob-Neumann-geral} are bounded if the growth of $f$ is as in \eqref{growth-critico}. Boundedness of solution for $p$-Laplace type equations with Dirichlet boundary condition and sources with critical growth has been proved, for instance, in Guedda and Veron \cite{Gued-Vero},  and Peral \cite{Pe}, see also the references therein. More general operators have been treated in Serrin \cite{Serr}, and Ladyzhenskaya and  Ural'tseva \cite{Lady-Ural}. For anisotropic $p$-Laplace equations in unbounded cones supplied with Neumann boundary condition, boundedness of solution has been obtained in  Ciraolo, Figalli and Roncoroni \cite{Figalli 1}. 

\smallskip

We first prove a technical lemma which will be useful in several calculations.

\begin{lemma}\label{lema1-limitacao}
Let $g:\Omega\rightarrow\mathbb{R}$ be a measurable function and $u \in W^{1,p}(\Omega)$ be such that $H(\nabla u)^{p-1}\langle \nabla H(\nabla u),\nabla u\rangle g$ is integrable. Then
$$\int_{\Omega}H(\nabla u)^{p-1}\langle \nabla H(\nabla u), \nabla u\rangle g(x)\,dx= \int_{\Omega}H(\nabla u)^{p}g(x)\,dx.$$
\end{lemma}
\begin{proof}
Since from \eqref{norm-eqv} and Lemma \ref{H-prop}(iii) one has $\langle\nabla H(\xi),\xi\rangle\rightarrow0$, as $\xi\rightarrow0,$ using again Lemma \ref{H-prop}(iii) it follows that

\medskip

\noindent\qquad$\int_{\Omega}H(\nabla u)^{p-1}\langle \nabla H(\nabla u), \nabla u\rangle g(x)\,dx= \int_{\{|\nabla u|>0\}}\!\!H(\nabla u)^{p}g(x)\,dx= \int_{\Omega}H(\nabla u)^{p}g(x)\,dx.$
\end{proof}

\medskip

Next result establishes integrability of all possible solution of \eqref{prob-Neumann-geral}. By adapting the proof of a theorem of Brezis and Kato \cite{BK}, see also Struwe \cite{Stru}, we prove the following
\begin{lemma}\label{lemma-integrabilidade-finita}
Let $\Omega\subset\mathbb{R}^N,$ $N\geqslant3,$ be a bounded smooth domain, and $f:\Omega\times\mathbb{R}\rightarrow\mathbb{R}$ be a Carath\'eodory function satisfying
$$|f(x,u)|\leqslant a(x)(1+|u|^{p-1}),\quad\mbox{a.e. }x\in\Omega,\:\forall u\in\mathbb{R},$$where $1<p<N$, and $a\in L^{\frac{N}{p}}(\Omega).$ Then each weak solution $u \in W^{1,p}(\Omega)$ of \eqref{prob-Neumann-geral} belongs to $L^q(\Omega)$ for all $q\in[1,\infty).$
\end{lemma}
\begin{proof}
The core in the proof is to establish the following
\begin{claim}\label{claim-integrab}
\textnormal{If $u\in L^{p(s+1)}(\Omega)$ for some $s\geqslant0$ then $|u|^{s+1}\in W^{1,p}(\Omega).$}
\end{claim}

\noindent Indeed, let us choose $\phi=u(|u|\wedge M)^{ps}$ as test function in \eqref{prob-Neumann-geral}, where $M>0$ and $|u|\wedge M=\min\{|u|,M\}.$ Thus we have
$$
\begin{aligned}
	\int_{\Omega}H(\nabla u)^{p-1}\langle \nabla H(\nabla u), \nabla \phi\rangle\,dx  &\leqslant \int_{\Omega}|a(x)|(1+|u|^{p-1})\phi\,dx\\
		&\leqslant\int_{\Omega}|a(x)|(1+2|u|^{p})(|u|\wedge M)^{ps}\,dx\\
	&\leqslant 3\int_{\Omega}|a(x)||u|^{p}(|u|\wedge M)^{ps}\,dx +\int_{\Omega}|a(x)|\,dx\\
		&\leqslant C+3\int_{\{a(x)\geqslant\eta\}}|a(x)|(|u|(|u|\wedge M)^{s})^{p}\,dx
\end{aligned}
$$where $\eta>0$ is arbitrary and $C=C\big(\|a\|_{L^1(\Omega)},\|u\|_{L^{p(s+1)}(\Omega)}\big)$ is a positive constant independent on $M$. From H\"older and Sobolev's inequalities, we obtain
$$
\begin{aligned}
	\int_{\Omega}H(\nabla u)^{p-1}\langle \nabla H(\nabla u), \nabla \phi\rangle\,dx &\leqslant C+3\bigg(\int_{\{a(x)\geqslant\eta\}}\!|a(x)|^{\frac{N}{p}}\,dx\bigg)^{\frac{p}{N}}\bigg(\int_{\Omega}(|u|(|u|\wedge M)^{s})^{p^*}\,dx\bigg)^{\frac{p}{p^*}}\\
	&\leqslant C+\widetilde{C}\bigg(\int_{\{a(x)\geqslant\eta\}}\!|a(x)|^{\frac{N}{p}}\,dx\bigg)^{\frac{p}{N}}\|\nabla(|u|(|u|\wedge M)^{s})\|_{L^{p}(\Omega)}^{p}
\end{aligned}$$where $\widetilde{C}>0$ is a constant independent on $M$. Therefore, on the one hand, the following estimate holds
\begin{equation}\label{est1-integrability}
\int_{\Omega}H(\nabla u)^{p-1}\langle \nabla H(\nabla u), \nabla \phi\rangle\,dx \leqslant C+\varepsilon(\eta)\|\nabla(|u|(|u|\wedge M)^{s})\|_{L^{p}(\Omega)}^{p},
\end{equation}where
$$\varepsilon(\eta):=\widetilde{C}\bigg(\int_{\{a(x)\geqslant\eta\}}\!|a(x)|^{\frac{N}{p}}\,dx\bigg)^{\frac{p}{N}}\longrightarrow0,\quad\mbox{as }\eta\rightarrow\infty.$$On the other hand, denoting $\mathcal{M}=\{x\in\Omega:|u(x)|\leqslant M\}$, we use Lemma \ref{lema1-limitacao} to calculate
$$
\begin{aligned}
	\int_{\Omega}\!\!H(\nabla u)^{p-1}\langle \nabla H(\nabla u), \nabla \phi\rangle&= 
	\!\int_{\Omega}\!\!H(\nabla u)^{p}(|u|\wedge M)^{ps}+ps\!\!\int_{\mathcal{M}}\!\!\!H(\nabla u)^{p-1}u|u|^{ps-1}\langle \nabla H(\nabla u), \nabla|u|\rangle
		\\	&= \int_{\Omega}H(\nabla u)^{p}(|u|\wedge M)^{ps}+ps\int_{\mathcal{M}}\!H(\nabla u)^{p}|u|^{ps}\\
		&=\!\int_{\Omega}\!\!H(\nabla u)^{p}(|u|\wedge M)^{ps}+ps\!\!\int_{\Omega}\!\!H(\nabla (|u|\wedge M))^{p}|u|^p(|u|\wedge M)^{p(s-1)}.
\end{aligned}
$$Thus from \eqref{norm-eqv} we can estimate
$$
\begin{aligned}
	\int_{\Omega}H(\nabla u)^{p-1}\langle \nabla H(\nabla u), \nabla \phi\rangle
		&\geqslant C\bigg[\int_{\Omega}(|u|\wedge M)^{ps}|\nabla u|^{p}+\int_{\Omega}|u|^p(|u|\wedge M)^{p(s-1)}|\nabla (|u|\wedge M)|^{p}\bigg]\\
		&\geqslant C\int_{\Omega}\big|(|u|\wedge M)^{s}\nabla|u|+s|u|(|u|\wedge M)^{(s-1)}\nabla (|u|\wedge M)\big|^{p}\,dx,
\end{aligned}$$where $C>0$ is a constant independent on $M$ and $\eta$, which leads to
\begin{equation}\label{est2-integrability}
\int_{\Omega}H(\nabla u)^{p-1}\langle \nabla H(\nabla u), \nabla \phi\rangle\,dx\geqslant C\|\nabla (|u|(|u|\wedge M)^s)\|^p_{L^p(\Omega)}.
\end{equation}Combining \eqref{est1-integrability} and \eqref{est2-integrability} it is possible to choose $\varepsilon(\eta)$ sufficiently small (by choosing $\eta$ large enough) to ensure $|u|(|u|\wedge M)^s$ is bounded in $W^{1,p}(\Omega),$ as $M\rightarrow\infty.$ This proves the claim.

Finally, since Claim \ref{claim-integrab} holds with $s=0$ and $1<p<N$, iterating Claim \ref{claim-integrab} we obtain $u\in L^{q_k}(\Omega),$ where $q_k=(N/(N-p))^k$ for all $k\geqslant1$. Since $q_k\rightarrow\infty,$ the lemma is proved. 
\end{proof}

We are ready to prove the main result in this section, which is the following 

\begin{theorem}\label{thm-L-infinito}
Let $\Omega\subset\mathbb{R}^N,$ $N\geqslant3,$ be a bounded smooth domain, and $f:\Omega\times\mathbb{R}\rightarrow\mathbb{R}$ be a Carath\'eodory function satisfying
\begin{equation}\label{growth-critico}
|f(x,u)|\leqslant C(1+|u|^{r-1}),\quad\mbox{a.e. }x\in\Omega,\:\forall u\in\mathbb{R},
\end{equation}where $1<p<N$, and $r\in(1,p^*].$ If $u \in W^{1,p}(\Omega)$ is a weak solution of \eqref{prob-Neumann-geral} then $u\in L^{\infty}(\Omega)$. 
\end{theorem}
\begin{proof}
It suffices to prove that each weak solution $u \in W^{1,p}(\Omega)$ of the problem
\begin{equation}\label{prob-aux-bdd}
\left\{\begin{array}{lcl}
	-\Delta ^H_pu+|u|^{p-2}u=g(x,u)\quad \mbox{in\,\,}\Omega,\\
	a(\nabla u)\cdot \nu =0\quad \mbox{on\,\,}\partial \Omega,
\end{array}\right.
\end{equation}belongs to $L^{\infty}(\Omega)$, where $g$ is a Carath\'eodory function satisfying \eqref{growth-critico}. Indeed, from \eqref{prob-aux-bdd} one recovers \eqref{prob-Neumann-geral} by setting $g(\cdot,u)=f(\cdot,u)+|u|^{p-2}u.$ Note that since $$|g(x,u)|\leqslant a(x)(1+|u|^{p-1}),\quad\mbox{a.e. }x\in\Omega,$$where $a(\cdot)=C\max\{1,|u(\cdot)|^{p^*-p}\}$, which in turn obeys $a\in L^{\frac{N}{p}}(\Omega),$ from Lemma \ref{lemma-integrabilidade-finita} we have $u\in L^{q}(\Omega)$ for all $q\in[1,\infty).$ In what follows, we shall prove that $u^+\in L^{\infty}(\Omega).$

Let $A_k\doteq\{x\in\Omega\,:\,u(x)\geqslant k\}$ for all $k>0.$ Choose $\sigma_1>1$ sufficiently large in such a way that
  $1/(\sigma_1(p-1))<1/p-1/p^*$, and let $\sigma=\sigma_1/(\sigma_1-1).$ Setting $\phi=(u-k)^+,$ note that
      $$
\begin{aligned}      
      \|\phi\|^p_{W^{1,p}(\Omega)}\leqslant C\int_{\Omega}\big[H(\nabla\phi)^p+|\phi|^p\big]\,dx&= C\int_{\Omega}\Big[H(\nabla u)^{p-1}\langle \nabla H(\nabla u), \nabla \phi\rangle+|\phi|^{p}\Big]\,dx\\ 
      &\leqslant C\int_{\Omega}\Big[H(\nabla u)^{p-1}\langle \nabla H(\nabla u), \nabla \phi\rangle+|u|^{p-2}u\phi\Big]\,dx.
\end{aligned}$$
 So using $\phi$ as test function in \eqref{prob-aux-bdd}, from H\"older's inequality and \eqref{norm-eqv} we have
  $$\|\phi\|^p_{W^{1,p}(\Omega)}\leqslant C \int_{\Omega}(1+|u|^{r-1})|\phi|\,dx\leqslant C\|\phi\|_{L^{\sigma}(\Omega)},$$
 where $C=C(\||u|^{r-1}\|_{L^{\sigma_1}(\Omega)})>0$ is a constant independent on $k$. Since $s<p,$ again by H\"older's inequality we have for all $k>0$
  $$\|\phi\|_{L^{\sigma}(\Omega)}\leqslant\|\phi\|_{L^{p}(\Omega)}|A_k|^{\frac{1}{s}-\frac{1}{p}}.$$Thus we can infer that
\begin{equation}\label{eq1-limitacao}
  \|\phi\|_{W^{1,p}(\Omega)}\leqslant C|A_k|^{\frac{1}{p-1}\big(\frac{1}{s}-\frac{1}{p}\big)},\quad\forall k>0.
\end{equation}Now by Sobolev's and H\"older's inequalities, we have 
\begin{equation}\label{eq2-limitacao}
\int_{\Omega}(u-k)^+\,dx\leqslant C  \|\phi\|_{W^{1,p}(\Omega)}|A_k|^{1-\frac{1}{p^*}},\quad\forall k>0.
\end{equation}From \eqref{eq1-limitacao} and \eqref{eq2-limitacao}, we obtain
\begin{equation}\label{eq3-limitacao}
\int_{\Omega}(u-k)^+\,dx\leqslant C |A_k|^{1+\delta},\quad\forall k>0,
\end{equation}where $\delta=1/p-1/p^*-1/(\sigma_1(p-1)),$ which is a positive constant from the choice of $\sigma_1.$ Such an estimate on level sets of the form \eqref{eq3-limitacao} ensures $u^+\in L^{\infty}(\Omega),$ see \cite{Lady-Ural} (Lemma 2.1, Chapter 5). The same argument also applies to show that $u^-\in L^{\infty}(\Omega).$ Hence $u\in L^{\infty}(\Omega)$.
 \end{proof}
\begin{remark}
\textnormal{Theorem \ref{thm-L-infinito} applies for general Neumann problems involving anisotropic $p$-Laplace equations with sources of critical growth on bounded Lipschitz domains. Indeed, the regularity necessary on $\Omega$ in Theorem \ref{thm-L-infinito} concerns the continuous embedding $W^{1,p}(\Omega)\hookrightarrow L^{p^*}(\Omega)
$, which holds for bounded Lipschitz domains, see \cite{Gilb-Trud}.}
\end{remark}
\section{Critical Level Estimates }

\begin{lemma}\label{L21}
	Let $\tilde{B}=B_1 \cap \{x_n > h(x')\},$ where $B_1=B(0,1)$ is the unit ball in $\R^N$. Let $h(x')$ be a $C^1$ function defined on $\{ x' \in \R^{N-1}, |x'|<1\}$ such that $h$ and $Dh$ vanish at $0$. Then for any $u \in W^{1,p}(B_1)$ with $\mbox{\textnormal{supp\,}} u \subset B_1$, the following holds.
	\begin{itemize}
		\item[{\bf (i)}] If $h \equiv 0$, then
		\begin{equation}\label{2}
			\int_{\tilde{B}}|H(\nabla u)|^pdx \ge 2^{\frac{-p}{n}}S_{\Sigma, H}\left[\int_{\tilde{B}}|u|^{p^*}dx\right]^{p/p^*}.
		\end{equation}
		\item[{\bf (ii)}] For all $\varepsilon>0$ there exists $\delta>0$ such that, if $|Dh|\le \delta,$ then
		\begin{equation}\label{3}
			\int_{\tilde{B}}|H(\nabla u)|^pdx \ge \left(2^{\frac{-p}{n}}S_{\Sigma, H}-\varepsilon\right)\left[\int_{\tilde{B}}|u|^{p^*}dx\right]^{p/p^*}.
		\end{equation} 
	\end{itemize} 
\end{lemma}
\begin{proof} {\bf (i)} Since the values of $u(x)$ are irrelevant for $x_n<0$, we may assume $u$ is even with respect to $x_n. $ So we have
$$
\begin{aligned}
\int_{\tilde{B}}|H(\nabla u)|^pdx=\frac{1}{2}\int_{B_1}|H(\nabla u)|^pdx\ge \frac{1}{2}S_{\Sigma, H}\bigg[\int_{B_1}|u|^{p^*}dx\bigg]^{p/p^*}\!\!&\geq \frac{1}{2}S_{\Sigma, H}\bigg[2\int_{\tilde{B}}|u|^{p^*}dx\bigg]^{p/p^*}\\
&\geq 2^{p/p^*-1}S_{\Sigma, H}\bigg[\int_{\tilde{B}}|u|^{p^*}dx\bigg]^{p/p^*}\\
		&= 2^{-p/n}S_{\Sigma, H}\bigg[\int_{\tilde{B}}|u|^{p^*}dx\bigg]^{p/p^*}.
\end{aligned}
$$
\noindent{\bf (ii)} Making the change of variable $y'=x'$ and $y_n=x_n-g(x')$, if $|1-g'(y')|\le \delta$ it follows that
$$
\begin{aligned}
	\int_{\tilde{B}}|H(\nabla u)|^pdx=\int_{\R^{N-1}}dx'\int_{-g(y')}^{0}|H(\nabla u)|^pdx_n &= \int_{\R^{N-1}}dy'\int_{0}^{g(x')}|H(\nabla u)|^p|1-g'(y')|dy_n\\
		&\ge \int_{\R^{N-1}}\!\!\!dy'\int_{0}^{g(y')}\!\!\!(1-|g'(y')|)|H(\nabla u)|^pdy_n\\
	&\ge \int_{\R^{N-1}}\!\!\!dy'\int_{0}^{g(y')}\!\!\!(1-\delta)|H(\nabla u)|^pdy_n\\
		&=(1-\delta)\int_{\tilde{B}}|H(\nabla u)|^pdx\\
&\ge (1-\delta)2^{-p/n}S_{\Sigma, H}\left[\int_{\tilde{B}}|u|^{p^*}dx\right]^{p/p^*}\\
&\ge (2^{-p/n}S_{\Sigma, H}-\varepsilon)\left[\int_{\tilde{B}}|u|^{p^*}dx\right]^{p/p^*}.
\end{aligned}
$$
\end{proof}From now on, let us denote by $\|\cdot\|$ the anisotropic norm on $W^{1,p}(\Omega)$ induced by the right-hand side of $(P),$ i.e., 
$$\|u\|=\bigg(\int_{\Omega}[H^p(\nabla u)+|u|^p]dx\bigg)^{1/p},\quad \forall u\in W^{1,p}(\Omega).$$
\begin{lemma}\label{MP Geometry}
The functional $J$ in \eqref{func-J} enjoys the  mountain pass geometry, i.e., it holds: 
\begin{itemize}
\item[{\bf (i)}] There exist $\theta, \beta>0$ such that $J(u)\geq \beta$ if one has $\|u\|= \theta.$
\item[{\bf (ii)}]There exists $e\in D^{1,p}(\R^N)$ with $\|e\|> \theta$ such that $J(e)<0.$
\end{itemize}
\end{lemma}

\begin{proof} {\bf (i)} From \eqref{func-J} and the continuous embedding $W^{1,p}(\Omega) \hookrightarrow L^{p*}(\Omega),$ we have
$$J(u)
	\ge\tilde{a}\frac{1}{p}\|u\|^p-\frac{\lambda}{q}\int_{\Omega}|u|^q-\frac{1}{p^*}\int_{\Omega}|u|^{p^*}dx
	\ge\tilde{a}\frac{1}{p}\|u\|^p-C_1\frac{\lambda}{q}\|u\|^q-\frac{C_2}{p^*}\|u\|^{p^*}.$$
Thus setting $\|u\|=\theta$ with $\theta>0$ sufficiently small, since $p<q<p^*-1$ it is possible to find $\beta>0$ such that $J(u)=\beta>0.$  {\bf (ii)} Let $t>0$ and $w \in C^{\infty}_{0}(\R^N)$, $w\not\equiv0,$ be given. From \eqref{func-J} we have
$$J(tw)\geq\tilde{b}\frac{1}{p}\|tw\|^p-\frac{\lambda}{q}\int_{\Omega}|tw|^q-\frac{1}{p^*}\int_{\Omega}|tw|^{p^*}dx\geq\tilde{b}t^p\frac{1}{p}\|w\|^p-\frac{t^q\lambda}{q}\int_{\Omega}|w|^q-\frac{t^{p^*}}{p^*}\int_{\Omega}|w|^{p^*}dx.$$Therefore $J(tw)\to -\infty,$ as $t \to +\infty$,  since $p<q<p^*-1$. Then it suffices to choose $e=t_0w$ with $t_0>$ sufficiently large to conclude the proof.
\end{proof}

\begin{lemma}\label{bound}
	If $(u_n)$ is a $(PS)_c$ sequence of $J$, then $(u_n)$ is bounded in $W^{1,p}(\Omega)$.
\end{lemma}
\begin{proof}

Let $(u_n)$ be a $(PS)_c$ sequence of $J$, so that $J(u_n) \to c$ and $J'(u_n) \to 0$. First, by setting $f(t)=\lambda|t|^{q-2}t+ |t|^{p^*-2}t$, we have $F(t)\big/f(t)t\rightarrow 1/p^*,$ as $t \to +\infty.$ Thus from \eqref{func-J} we obtain
\begin{equation}\label{bound1}
J(u_n)-\frac{1}{p^*}J'(u_n)u_n\ge \tilde{a}\left[\frac{1}{p}-\frac{1}{p^*}\right]\|u_n\|^p-k,
\end{equation}
where $k>0$ is a constant independent on $n$. Now a direct calculation gives 
\begin{equation}\label{bound2}
	J(u_n)-\frac{1}{p^*}J'(u_n)u_n \le c+o_n(1)\|u_n\|,
\end{equation} where $o_n(1)\rightarrow0,$ as $n\rightarrow+\infty$. 
Combining (\ref{bound1}) and (\ref{bound2}),  it follows that $(u_n)$ is bounded in $W^{1,p}(\Omega).$
\end{proof} 

%
%
\begin{lemma}\label{RM1}
	Suppose $c<\frac{1}{2N\tilde{C}}S_{\Sigma,H}^{N/p}.$  Then there is a solution $u$ of $(P)$ satisfying $J(u)\leq c$.
\end{lemma}
\begin{proof}
By Mountain Pass Theorem without $(PS)$ condition, there exists a sequence $(u_n)$ such that $J(u_n) \to c$ and $J'(u_n) \to 0$. Thanks to Lemma \ref{bound}, $(u_n)$ is bounded and, up to a subsequence if necessary, by the Sobolev embedding and Proposition \ref{Msec2} we have
\begin{eqnarray}
	u_n \rightharpoonup u\,\, \mbox{in}\,\, W^{1,p}(\Omega)\\
	u_n \rightarrow u\,\,\mbox{in}\,\,L^{q}(\Omega),\,\,1\le q< p^* \label{convlq}\\
	\nabla u_n \rightarrow \nabla u \,\,\mbox{in}\,\,(L^{q}(\Omega))^N,\,\,\,\mbox{for any}\,\,q<p.\\
	u_n \rightarrow u\,\,a.e.\,\,\mbox{in}\,\,\Omega.
\end{eqnarray}Then a direct calculation shows that
\begin{equation}
\int_{\Omega} \left[H(\nabla u)^{p-1}\langle \nabla H(\nabla u), \nabla v\rangle+|u|^{p-2}uv \right]dx -\lambda\int_{\Omega} (u^+)^{q-1} v dx-\int_{\Omega}(u^+)^{p^*-1}v dx=0,
\end{equation}for all $v\in W^{1,p}(\Omega),$ 
that is, $u$ is a critical point of $J$. To see that $u \not \equiv 0,$ we first recall that
\begin{enumerate}
	\item[(i)] $J'(u_n)u_n=\int_{\Omega} \left[H(\nabla u_n)^{p}+|u_n|^{p} \right]dx - \lambda\int_{\Omega} (u_n^+)^{q} dx-\int_{\Omega}(u_n^+)^{p^*}dx \to 0;$
	\item[(ii)] $J(u_n)=\frac{1}{p}\int_{\Omega} \left[H(\nabla u_n)^{p}+|u_n|^{p} \right]dx - \frac{\lambda}{q}\int_{\Omega} (u_n^+)^{q} dx-\frac{1}{p^*}\int_{\Omega}(u_n^+)^{p^*}dx \to c.$
\end{enumerate}
Suppose, by contradiction, $u \equiv 0 $. We have
$$\int_{\Omega} \left[H(\nabla u_n)^{p}+|u_n|^{p} \right]dx \rightarrow l > 0$$ since $c>0$. From (\ref{convlq}) we obtain
$$\int_{\Omega} (u_n^+)^{q} dx \rightarrow 0 \quad\mbox{and}\quad\int_{\Omega}(u_n^+)^{p^*}dx\rightarrow l,$$
what implies
\begin{equation}\label{lc}
	\left(\frac{1}{p}-\frac{1}{p^*}\right)l=c.
\end{equation}
Let $\varepsilon>0$ be given. Choose $(\varphi_{\alpha})_{\alpha=1}^{m}$ a unity of partition on $\overline{\Omega}$ with diameter, $diam(\mbox{supp\,} \varphi_{\alpha})<\delta,$ for each $\alpha$. Since $\Omega$ is a domain of class  $C^1$, we have
$$\int_{\Omega}\left[|H(\nabla(u\varphi_{\alpha}))|^p+|u \varphi_{\alpha}|^p\right]dx \ge \left(2^{-p/N}S_{\Sigma, H}-\varepsilon\right)\int_{\Omega}|u \varphi_{\alpha}|^{p^*}dx,$$
for all $\alpha=1,2,\cdots,m,$ and all $u\in W^{1,p}(\Omega).$ Therefore
\begin{equation}\label{emb1}
	\begin{array}{lcl}
		\left[\int_{\Omega}|u_n|^{p^*}dx\right]^{p/p^*}&\le& \left[\int_{\Omega}\sum_{\alpha=1}^{m}\varphi_{\alpha}^{p^*/p}|u_n|^{p^*}dx\right]^{p/p^*}\le \sum_{\alpha=1}^{m}\left[\int_{\Omega}|\varphi_{\alpha}^{1/p}u_n|^{p^*}dx\right]^{p/p^*}\\
		&\le&\left(2^{-p/N}S_{\Sigma, H}-\varepsilon\right)^{-1}\sum_{\alpha=1}^{m}\left[\int_{\Omega}|H(\nabla(\varphi_{\alpha}^{1/p}u_n))|^{p}dx\right]\\
		&\le&\left(2^{-p/N}S_{\Sigma, H}-\varepsilon\right)^{-1}C_1\sum_{\alpha=1}^{m}\left[\int_{\Omega}|\nabla(\varphi_{\alpha}^{1/p}u_n)|^{p}dx\right].
	\end{array}
\end{equation}
Now on the one hand, since $v=\sum_{\alpha=1}^{m}\varphi_{\alpha} v$ and $\sum_{\alpha=1}^{m}\varphi_{\alpha} =1$, we obtain
$$
\begin{array}{lcl}
	\sum_{\alpha=1}^{m}\left[\int_{\Omega}|\nabla(\varphi_{\alpha}^{1/p}u_n)|^{p}dx\right]&=&\ \sum_{\alpha=1}^{m}\int_{\Omega}|\varphi_{\alpha}^{1/p}\nabla u_n+u_n \nabla(\varphi_{\alpha}^{1/p})|^{p}dx\\
	&\le & 2^{p-1}\sum_{\alpha=1}^{m} \int_{\Omega} \left(|\varphi_{\alpha}||\nabla u_n|^p+|u_n|^pp^{-p} \varphi_{\alpha}^{1-p}|\nabla\varphi_{\alpha}|^{p}\right)dx\\
	&\le &C \int_{\Omega} \left(\sum_{\alpha=1}^{m}\varphi_{\alpha}|\nabla u_n|^p+\sum_{\alpha=1}^{m}|u_n|^p \varphi_{\alpha}\left(\frac{|\nabla\varphi_{\alpha}|}{\varphi_{\alpha}}\right)^p\right)dx\\
	&\le &C \int_{\Omega} \Big(|\nabla u_n|^p+\sum_{\alpha=1}^{m}|u_n|^p \varphi_{\alpha}\Big)dx\\
	&\le &C \int_{\Omega} \left(|\nabla u_n|^p+|u_n|^p \right)dx.
\end{array}$$
On the other hand, by Young's inequality we have
\begin{equation}\label{emb2}
	\begin{array}{lcl}
		\sum_{\alpha=1}^{m}\Big[\int_{\Omega}|\nabla(\varphi_{\alpha}^{1/p}u_n)|^{p}dx\Big] &\le &C(1+\varepsilon) \int_{\Omega} |\nabla u_n|^pdx+C(\varepsilon)\int_{\Omega}|u_n|^pdx\\
		&\le&C(1+\varepsilon) \int_{\Omega} |\nabla u_n|^pdx+o(n)\\
		&\le &\tilde{C}(1+\varepsilon) \int_{\Omega} |H(\nabla u_n)|^pdx+o(n).
	\end{array}
\end{equation}
Combining (\ref{emb1}) and (\ref{emb2}), it follows that
\begin{equation}\label{emb3}
\left[\int_{\Omega}|u_n|^{p^*}dx\right]^{p/p^*}\le\tilde{C}(1+\varepsilon)\left(2^{-p/N}S_{\Sigma, H}-\varepsilon\right)^{-1} \int_{\Omega} \big[|H(\nabla u_n)|^p+| u_n|^p\big]dx+o(n).
\end{equation}
Letting $n\rightarrow+\infty$, we obtain 
$$l\ge \tilde{C}^{-1}\left[\frac{\left(2^{-p/N}S_{\Sigma, H}-\varepsilon\right)}{(1+\varepsilon)}\right]^{n/p}.$$Taking (\ref{lc}) into account,  we deduce
\begin{equation}\label{cestimate}
	c> \tilde{C}^{-1}\frac{1}{N}\left[\frac{\left(2^{-p/N}S_{\Sigma, H}-\varepsilon\right)}{(1+\varepsilon)}\right]^{N/p}.
\end{equation}Letting $\varepsilon\rightarrow0,$ we obtain from \eqref{cestimate}
$$c\le\frac{1}{2N}\left(\frac{S_{\Sigma, H}}{\tilde{C}}\right),$$a contradiction to the choice of $c$. Hence 
$u\not \equiv0.$ Finally, to see that $J(u)\leq c,$ choosing  $v_n=u_n-u$ in $J'(u_n)\rightarrow0$,  from Brezis-Lieb lemma we have 
$$
\begin{aligned}
\int_{\Omega} \left[H(\nabla u_n)^{p}+|u_n|^{p} \right]dx&=\int_{\Omega} \left[H(\nabla v_n)^{p}+|v_n|^{p} \right]dx + \int_{\Omega} \left[H(\nabla u)^{p}+|u|^{p} \right]dx+o(1)\\
\int_{\Omega}|u_n|^{p^*}dx&=\int_{\Omega}|v_n|^{p^*}dx+\int_{\Omega}|u|^{p^*}dx+o(1).
\end{aligned}
$$From (ii) we have
$$J(u)+\frac{1}{p}\int_{\Omega} \left[H(\nabla v_n)^{p}+|v_n|^{p} \right]dx-\frac{1}{p^*}\int_{\Omega}|v_n|^{p^*}dx=c+o(1),$$so, since
$$\int_{\Omega} \left[H(\nabla v_n)^{p}+|v_n|^{p} \right]dx-\int_{\Omega}|v_n|^{p^*}dx=o(1),$$it follows that
$$J(u)+\left(\frac{1}{p}-\frac{1}{p^*}\right)\int_{\Omega} \left[H(\nabla v_n)^{p}+|v_n|^{p} \right]dx=c+o(1).$$Therefore $J(u)\leq c,$ and the proof is complete.
\end{proof}

In the next lemma we collect some estimates and  asymptotic limits on quantities involving the extremal functions \eqref{extrem-funct}. 
\begin{lemma}\label{Ident} Set the extremal function $u_{\varepsilon}(x)=U^{H}_{\varepsilon{\frac{p-1}{p}}, x_0}(x).$ Then the following estimates and asymptotic limits hold.
\begin{eqnarray}
K_1(\varepsilon)&:=&\int_{\Omega}|H(\nabla u_{\varepsilon})|^pdx=\int_{\R^{N}_{+}}|H(\nabla u_{\varepsilon})|^pdx-I_\varepsilon+o(\varepsilon^{\frac{p-1}{p}})\\
	K_2(\varepsilon)&:=&\int_{\Omega}| u_{\varepsilon}|^{p^*}dx=\int_{\R^{N}_{+}}| u_{\varepsilon}|^{p^*}dx-II_\varepsilon+O(\varepsilon^{\frac{p-1}{p}})\\\nonumber\\
	K_3(\varepsilon)&:=&\int_{\Omega}| u_{\varepsilon}|^{p}dx=\left\{\begin{array}{lcl}
	O(|\varepsilon^{\frac{N-p}{p}}\mbox{\textnormal{ln\,}} \varepsilon|)=O(|\varepsilon^{p-1}\mbox{\textnormal{ln\,}} \varepsilon|),\, N=p^2,\\ 
	O(\varepsilon^{\frac{N-p}{p}})=O(\varepsilon^{\frac{(p-1)^2}{p}}),\, N=p^2-p+1,\\
	O(\varepsilon^{p-1}),\, N>p^2,\\
\end{array}
	\right.
\end{eqnarray}
where 
\begin{equation}
	I_\varepsilon=\int_{\R^{N-1}}dx'\int_{0}^{g(x')}|H(\nabla u_{\varepsilon})|^pdx_n\quad\mbox{and}\quad II_\varepsilon=\int_{\R^{N-1}}dx'\int_{0}^{g(x')}| u_{\varepsilon}|^{p^*}dx_n.
\end{equation}
Furthermore,
\begin{equation}\label{Iest1}
	\lim_{\varepsilon \to 0} \varepsilon^{-(\frac{p-1}{p})}I_\varepsilon=\left(\frac{n-p}{p-1}\right)^p\int_{\R^{N-1}} \frac{|y'|^{p/p-1}g(y')dy'}{\left(1+|y'|^{p/p-1} \right)^N}
\end{equation}

\begin{equation}
	\lim_{\varepsilon \to 0} \varepsilon^{-(\frac{p-1}{p})}II_\varepsilon=\int_{\R^{N-1}} \frac{g(y')dy'}{\left(1+|y'|^{p/p-1}\right)^N}
\end{equation}
\end{lemma}
\begin{proof}
	Since there exist $a,b>0$ such that $a|\xi|\le H(\xi)\le b|\xi|$ for all $\xi\in\mathbb{R}^N$, the proof of Lemma \ref{Ident} follows the same lines in the arguments in \cite{Wang,Kou}, and will be omitted.  
\end{proof}

\begin{remark}
	Note that
	$$\frac{K_1}{[K_2]^{p/p^*}}=S_{\Sigma, H},$$
	where $K_1=\int_{\R^N}|H(\nabla u_{\varepsilon})|^pdx$ and $K_2=\int_{\R^N}|u_{\varepsilon}|^{p^*}dx$, then
	$$K_1(\varepsilon)=\frac{1}{2}K_1-I_{\varepsilon}+o(\varepsilon^{p-1/p}),$$
	$$K_2(\varepsilon)=\frac{1}{2}K_2-II_{\varepsilon}+O(\varepsilon^{p-1/p}).$$
\end{remark}

\begin{lemma}\label{RK2}
	Assume the hypotheses in Theorem $\ref{Main}$ hold. Then there exists at least one non-negative function $u \in W^{1,p}(\Omega)\setminus\{0\}$ such that
	\begin{equation}\label{Main estimate}
		\sup_{t\ge 0}J(tu)<\frac{1}{2N\tilde{C}}S_{\Sigma,H}^{N/p}.
	\end{equation}
	
\end{lemma}
\begin{proof}
	Let
	 $Y_{\varepsilon}=J(t_{\varepsilon}u_{\varepsilon})=\sup_{t>0}J(tu_{\varepsilon}).$ Then we have
$$
\begin{aligned}
Y_{\varepsilon}&=\sup_{t>0}\left[\frac{1}{p}\int_{\Omega} \left[H(\nabla (tu_{\varepsilon}))^{p}+(tu_{\varepsilon})^{p} \right]dx - \frac{\lambda}{q}\int_{\Omega} t^qu_{\varepsilon}^{q} dx-\frac{1}{p^*}\int_{\Omega}t^{p^*}u_{\varepsilon}^{p^*}dx\right]\\
&=\sup_{t>0}\left[\frac{t^p}{p}\left(K_1(\varepsilon)+K_3(\varepsilon)\right)-\frac{\lambda t^q}{q} u_\varepsilon^qdx -\frac{t^{p^*}}{p^*}K_2(\varepsilon)\right]\\
&\le \frac{1}{N}\left[\frac{K_1(\varepsilon)+K_3(\varepsilon)}{K_2(\varepsilon)^{\frac{N-p}{N}}}\right]^{N/p}\!\!-\lambda\frac{C_0^q}{q}\int_{\Omega}u_{\varepsilon}^qdx,
\end{aligned}
$$for all $C_0<t\le C_1,$ for some $C_0,C_1>0$. Such $C_1$ there exists because of the geometry of $J$, and the existence of $C_0$ is true, otherwise, we could find $\varepsilon_n \to 0$  with $t_{\varepsilon_n} \to 0$, as $n \to \infty$. So, up to a subsequence, we have  $t_{\varepsilon_n}u_{\varepsilon_n} \to 0$ and
$$0<c\le \sup_{t>0}J(tu_{\varepsilon})=J(t_{\varepsilon_n}u_{\varepsilon_n})\to 0,$$
which is a contradiction. We shall split the proof into two cases. 
 
\bigskip

\noindent\mbox{\scriptsize$\bullet$}\quad\underline{Case  $N \ge p^2$}. We have $K_3(\varepsilon)=O(\varepsilon^{\frac{p-1}{p}}),$ 
and there exists $\varepsilon_0>0$ such that
$$K_2(\varepsilon) \le C_1,\,\,K_1(\varepsilon)+K_3(\varepsilon) \le C_2,\quad\forall \varepsilon \in (0, \varepsilon_0).$$
Therefore
$$Y_{\varepsilon}\le\frac{1}{N}\left[\frac{K_1(\varepsilon)}{K_2(\varepsilon)^{\frac{N-p}{N}}}\right]^{N/p}-\lambda\frac{C_0^q}{q}\int_{\Omega}u_{\varepsilon}^qdx.$$
\begin{claim}\textnormal{The following estimate holds}
	\begin{equation}\label{k-estimate}
	\frac{K_1(\varepsilon)}{K_2(\varepsilon)^{\frac{N-p}{N}}} < 2^{-p/N}S_{\Sigma,H}+o(\varepsilon^{\frac{p-1}{p}}).
	\end{equation}
\end{claim}
\noindent Indeed, since 
$$\left[\frac{K_1}{K_2^{\frac{N-p}{N}}}\right]^{N/p}=S_{\Sigma,H}^{\frac{N}{p}},$$
we obtain
\begin{equation}
	\frac{K_1(\varepsilon)}{K_2(\varepsilon)^{\frac{N-p}{N}}} < \frac{1}{2}\frac{K_1}{\left(\frac{1}{2}K_2\right)^{\frac{N-p}{N}}}+o(\varepsilon^{\frac{p-1}{p}}),
\end{equation}
in a such way that
$$K_1(\varepsilon)\left(\frac{1}{2}K_2\right)^{\frac{N-p}{N}}<\frac{1}{2}K_1K_2(\varepsilon)^{\frac{N-p}{N}}+o(\varepsilon^{\frac{p-1}{p}}).$$
By Lemma \ref{Ident}, we have
\begin{equation}\label{est-k}
\qquad\quad\bigg[\frac{1}{2}K_1-I_{\varepsilon}+o(\varepsilon^{\frac{p-1}{p}})\bigg]\Big(\frac{1}{2}K_2\Big)^{\frac{N-p}{N}}\!\!<\frac{1}{2}K_1\left[\frac{1}{2}K_2-II_{\varepsilon}+o(\varepsilon^{\frac{p-1}{p}})\right]^{\frac{N-p}{p}}\!\!+o(\varepsilon^{\frac{p-1}{p}}).
\end{equation}
Now using the inequality $(a-b)^{\alpha}\le a^{\alpha}-\alpha a^{\alpha-1}b$ for all $b\leq a,$ and $\alpha<1,$ we estimate 
$$
\begin{aligned}
\left[\frac{1}{2}K_2-II_{\varepsilon}+o(\varepsilon^{\frac{p-1}{p}})\right]^{\frac{p}{p^*}} &\le \left(\frac{1}{2}K_2-II_{\varepsilon}\right)^{\frac{p}{p^*}}+o(\varepsilon^{\frac{p-1}{p}})\\
&\le \left(\frac{1}{2}K_2\right)^{\frac{p}{p^*}}-\frac{p}{p^*}\left(\frac{1}{2}K_2\right)^{\frac{p}{p^*}-1}II_{\varepsilon}+o(\varepsilon^{\frac{p-1}{p}})\\
&= \left(\frac{1}{2}K_2\right)^{\frac{p}{p^*}}-\frac{p}{p^*}\left(\frac{1}{2}K_2\right)^{\frac{-p}{N}}II_{\varepsilon}+o(\varepsilon^{\frac{p-1}{p}}).
\end{aligned}
$$
Substitution into \eqref{est-k} gives 
$$\left[-I_{\varepsilon}+o(\varepsilon^{\frac{p-1}{p}})\right]\left(\frac{1}{2}K_2\right)^{\frac{N-p}{N}}<-\frac{1}{2}K_1\left[\frac{p}{p^*}\left(\frac{1}{2}K_2\right)^{\frac{-p}{N}}II_{\varepsilon}\right]+o(\varepsilon^{\frac{p-1}{p}}),$$
which reduces to
	\begin{equation}\label{K-estimate2}\frac{I_{\varepsilon}}{II_{\varepsilon}}> \left(\frac{N-p}{N}\right)\frac{K_1}{K_2}+o(1).
\end{equation}
Notice that 
$$
\begin{aligned}
\lim_{\varepsilon \to 0}\frac{I_{\varepsilon}}{II_{\varepsilon}}=\lim_{\varepsilon \to 0}\frac{\varepsilon^{-\frac{p-1}{p}} I_{\varepsilon}}{\varepsilon^{-\frac{p-1}{p}}II_{\varepsilon}}&=\left(\frac{N-p}{p-1}\right)^p\frac{\int_{\R^{N-1}}\frac{g(y')|y'|^{p/p-1}}{(1+|y'|^{p/p-1})^N}dy'}{\int_{\R^{N-1}}\frac{g(y')}{(1+|y'|^{p/p-1})^N}dy'}\\
&=\left(\frac{N-p}{p-1}\right)^p\frac{\frac{1}{2(N-1)}\sum_{i=1}^{N-1}\alpha_i\int_{\R^{N-1}}\frac{|y'|^{\frac{p}{p-1}+2}}{(1+|y'|^{p/(p-1)})^N}dy'}{\frac{1}{2(N-1)}\sum_{i=1}^{N-1}\alpha_i\int_{\R^{N-1}}\frac{|y'|^{2}}{(1+|y'|^{p/p-1})^N}dy'}.
\end{aligned}
$$
Using polar coordinates, we calculate 
$$\lim_{\varepsilon \to 0}\frac{I_{\varepsilon}}{II_{\varepsilon}}=\left(\frac{N-p}{p-1}\right)^p\frac{\int_{0}^{+\infty}\frac{r^{\frac{p}{p-1}+N}}{(1+r^{p/(p-1)})^N}dr}{\int_{0}^{+\infty}\frac{r^{N}}{(1+r^{p/(p-1)})^N}dr}.$$ Now for all $k\in \R$ with $\frac{p}{p-1}\le k \le \frac{p}{p-1}N,$ integrating by parts, we obtain 
\begin{equation}\label{int-est1}
	\int_{0}^{+\infty}\frac{r^k}{\left(1+r^{\frac{p}{p-1}}\right)^N}dr=\int_{0}^{+\infty}\frac{r^{k-\frac{p}{p-1}}}{\left(1+r^{\frac{p}{p-1}}\right)^{N-1}}dr-\int_{0}^{+\infty}\frac{r^{k-\frac{p}{p-1}}}{\left(1+r^{\frac{p}{p-1}}\right)^{N}}dr,
\end{equation}and 
\begin{equation}\label{int-est2}
	\int_{0}^{+\infty}\frac{r^{k-\frac{p}{p-1}}}{\left(1+r^{\frac{p}{p-1}}\right)^{N-1}}dr=\frac{p(N-1)}{(p-1)k-1}\int_{0}^{+\infty}\frac{r^{k}}{\left(1+r^{\frac{p}{p-1}}\right)^{N}}dr.
\end{equation}
Replacing (\ref{int-est2}) in (\ref{int-est1}), it follows that
\begin{equation}\label{int-est3}
	\int_{0}^{+\infty}\frac{r^k}{\left(1+r^{\frac{p}{p-1}}\right)^N}dr=\frac{(p-1)k-1}{pN-(p-1)-(p-1)k}\int_{0}^{+\infty}\frac{r^{k-\frac{p}{p-1}}}{\left(1+r^{\frac{p}{p-1}}\right)^{N}}dr.
\end{equation}
Choosing $k=N+\frac{p}{p-1}$ in \eqref{int-est3}, we achieve
$$\frac{\int_{0}^{+\infty}\frac{r^{N+\frac{p}{p-1}}}{\left(1+r^{\frac{p}{p-1}}\right)^N}dr}{\int_{0}^{+\infty}\frac{r^{N}}{\left(1+r^{\frac{p}{p-1}}\right)^{N}}dr}=\frac{(p-1)(N+1)}{N-(2p-1)}.$$
Therefore
\begin{equation}\label{I/II parte2}
\lim_{\varepsilon \to 0}\frac{I_{\varepsilon}}{II_{\varepsilon}}=\frac{(N-p)^p(N+1)}{(p-1)^{p-1}\left[N-(2p-1)\right]}.
\end{equation}
On the other hand,
$$
\begin{aligned}
\frac{(N-p)K_1}{nK_2}&=\frac{N-p}{N}\left(\frac{N-p}{p-1}\right)^p\frac{\int_{\R^{N}}\frac{|y|^{\frac{p}{p-1}}}{(1+|y|^{p/(p-1)})^N}dy}{\int_{\R^{N}}\frac{dy}{(1+|y|^{p/(p-1)})^N}}\\
	&=\frac{N-p}{N}\left(\frac{N-p}{p-1}\right)^p\frac{\int_{0}^{+\infty}\frac{r^{\frac{p}{p-1}}+N-1}{(1+r^{p/(p-1)})^N}dr}{\int_{0}^{\infty}\frac{r^{N-1}}{(1+r^{p/(p-1)})^N}dr}.
\end{aligned}
$$Choosing now $k=N+\frac{p}{p-1}-1$ in (\ref{int-est3}), we obtain
$$\frac{(N-p)K_1}{nK_2}=\frac{\left(N-p\right)^p}{(p-1)^{p-1}},$$ 
so, from \eqref{I/II parte2}, it follows that \eqref{K-estimate2} holds for all $N \ge p^2.$

\bigskip

\noindent\mbox{\scriptsize$\bullet$}\quad\underline{Case  $N=p^2-p+1$}. Let $0<\alpha\le A<\infty$ be such that $\alpha |x'|^2\le h(x')\le A|x'|^2,$ for all $\,x' \in D(0,\delta).$ Then
$$
\begin{aligned}
K_1(\varepsilon)&=\int_{\R^{N}_{+}}|\nabla u_{\varepsilon}|^pdx-\int_{D(0,\delta)}dx'\int_{0}^{h(x')}|\nabla u_\varepsilon|^p dx_N +O(\varepsilon^{\frac{N-p}{p}})\\
	&\le \frac{1}{2}K_1-\int_{D(0,\delta)}dx'\int_{0}^{h(x')}|\nabla u_\varepsilon|^p dx_N +O(\varepsilon^{\frac{p-1}{p}}).
\end{aligned}
$$Now we estimate
$$
\begin{aligned}
\int_{D(0,\delta)}dx'\int_{0}^{a|x'|^2}|\nabla u_\varepsilon|^p dx_N &\ge C \varepsilon^{\frac{N-p}{p}}\int_{D(0,\delta)}\frac{a|x'|^{\frac{p}{p-1}+2}}{(\varepsilon+|x'|^{\frac{p}{p-1}})^N}dx'\\
	&= C \varepsilon^{-N}\varepsilon^{\frac{N-p}{p}}\int_{D(0,\delta)}\frac{|x'|^{\frac{p}{p-1}+2}}{(1+|\frac{x'}{\varepsilon^{(p-1)/p}}|^{\frac{p}{p-1}})^N}dx'\\
		&= C \varepsilon^{-N+\frac{N-p}{p}}\int_{0}^{\frac{\delta}{\varepsilon^{\frac{p}{p-1}}}}\frac{\varepsilon^{\frac{p-1}{p}\left(\frac{p}{p-1}+2\right)}r}{(1+r^{\frac{p}{p-1}})^N}\varepsilon^{\frac{p-1}{p}(N-1)}dr=C\varepsilon^{\frac{p-1}{p}}|\mbox{\textnormal{ln\,}} \varepsilon|.
\end{aligned}
$$Therefore we obtain
\begin{equation}
	K_1(\varepsilon)=\frac{1}{2}K_1-C\varepsilon^{\frac{p-1}{p}}|\mbox{\textnormal{ln\,}} \varepsilon|+ O(\varepsilon^{\frac{p-1}{p}}).
\end{equation}
Analogously,
$$
\begin{aligned}
K_2(\varepsilon)&=\frac{1}{2}K_2-\int_{D(0,\delta)}dx'\int_{0}^{h(x')}| u_\varepsilon|^{p^*} dx_N +O(\varepsilon^{\frac{N}{p}})\\
	&\le  \frac{1}{2}K_2-\int_{D(0,\delta)}A\frac{\varepsilon^{\frac{N}{p}}|x'|^{2}}{(\varepsilon+|x'|^{\frac{p}{p-1}})^N}dx' +O(\varepsilon^{\frac{N}{p}})\\
	&=\frac{1}{2}K_2-O(\varepsilon^{\frac{p-1}{p}}).
\end{aligned}
$$So, similarly as done in the case $N \ge p^2,$ we have
$$
\begin{aligned}
	Y_{\varepsilon}=\sup_{t>0}J(tu_{\varepsilon})=J(t_{\varepsilon}u_{\varepsilon})
	\le\, &\, \sup_{t\ge 0}\left[\frac{t^p}{p}K_1(\varepsilon)
	-\frac{t^{p^*}}{p^*}K_2(\varepsilon)\right]-\lambda\frac{C_0^q}{q}\int_{\Omega}u_{\varepsilon}^qdx+O(\varepsilon^{\frac{p-1}{p}})\\
	&= \frac{1}{N}\left[\frac{K_1(\varepsilon)}{K_2(\varepsilon)^{\frac{N-p}{N}}}\right]^{N/p}-C\lambda \varepsilon^{\frac{N-p}{p^2}q}+O(\varepsilon^{\frac{p-1}{p}}).
\end{aligned}
$$Notice that
\begin{equation}\label{int-est 4}
	\frac{K_1(\varepsilon)}{K_2(\varepsilon)^{\frac{N-p}{N}}} < 2^{\frac{-p}{N}}S_{\Sigma,H}+O(\varepsilon^{\frac{p-1}{p}}).
\end{equation}Indeed, this follows since
$$\frac{1}{2}K_1-C\varepsilon^{\frac{p-1}{p}}|ln \varepsilon|<2^{\frac{-p}{N}}S_H\left[\frac{1}{2}K_2-O(\varepsilon^{\frac{p-1}{p}})\right]^{\frac{N-p}{N}}+O(\varepsilon^{\frac{p-1}{p}})=\frac{1}{2}S_{\Sigma,H}K_2^{\frac{N-p}{N}}+O(\varepsilon^{\frac{p-1}{p}}).$$
Taking into account that $S_H=\frac{K_1}{K_2^{\frac{N-p}{N}}},$ and $\frac{N-p}{p^2}q-\frac{N}{p}<0$ for $p<q<p^*,$
we obtain
$$	Y_{\varepsilon} \le \frac{1}{N}S_{\Sigma,H}^{\frac{N}{p}}+O(\varepsilon^{\frac{p-1}{p}})$$The proof is complete.
\end{proof}

\section{Proof of Theorem \ref{Main}}

\noindent \textit{Proof of Theorem $\ref{Main}$.} Set
$$c^{*}=\inf_{v\in W^{1,p}(\Omega)\setminus \{0\}}\left\{\sup_{t\ge 0}J(tv)\right\},$$
where $c^{*} > c,$ and $c$ denotes the mountain pass level. From Lemma \ref{RK2}, we have $c<\frac{1}{2N\tilde{C}}S_{\Sigma,H}^{N/p}.$ Hence, Lemma \ref{RM1} ensures there exists at least one non-trivial solution $u\in W^{1,p}(\Omega)$ of $(P)$. From a standard argument, testing $(P)$ with  $u^-$ implies $u \geqslant 0$. Applying Theorem \ref{thm-L-infinito} we have $u$ is a bounded non-negative solution of $(P)$. Then, arguing as in \cite{PPS}, it is possible to use classical elliptic regularity theory \cite{T} and the Harnack inequality established in \cite{Trudinger} to conclude that $u \in C^{1,\alpha}(\Omega)$ and $u>0$ in $\Omega$. The proof is complete. \hfill$\Box$

\end{document}